\documentclass[11pt]{article}
\title{Low-dimensional irreducible rational representations of classical algebraic groups}

\usepackage[numbers]{natbib}
\usepackage[hyphens]{url}
\usepackage{breakurl}
\usepackage{hyperref}
\usepackage{booktabs}
\usepackage[british]{babel}

\usepackage{amsmath}
\usepackage{amsfonts}
\usepackage{amssymb}

\usepackage[utf8]{inputenc}
\usepackage[parfill]{parskip}
\usepackage[margin=1in]{geometry}
\usepackage{setspace}
\usepackage{enumitem}
\usepackage{tabularx}
\usepackage{float}
\usepackage{mathrsfs}

\usepackage{amsthm}
\usepackage[many]{tcolorbox}

\newtheorem{theorem}{Theorem}[section]
\newtheorem{lemma}{Lemma}[section]
\newtheorem{proposition}[lemma]{Proposition}
\newtheorem{corollary}[lemma]{Corollary}
\theoremstyle{definition}

\theoremstyle{remark}
\newtheorem{remark}[lemma]{Remark}

\def \l {\lambda}
\def \a {\alpha}
\def \V {\tilde{V}}
\def \L {\tilde{L}}
\def \Z {\mathbb{Z}}

\DeclareMathOperator{\ch}{ch}
\DeclareMathOperator{\Div}{div}
\DeclareMathOperator{\rad}{rad}
\DeclareMathOperator{\Char}{char}
\DeclareMathOperator{\soc}{soc}

\begin{document}
\author{\'Alvaro L. Mart\'inez}
\date{\today}
\maketitle

\begin{abstract}
\noindent Let $G$ be an algebraic group of classical type of rank $l$ over an algebraically closed field $K$ of characteristic $p$. We list and determine the dimensions of all irreducible $KG$-modules $L$ with $\dim L < \binom{l+1}{4}$ if $G$ is of type $A_l$, and with $\dim L < 16 \binom{l}{4}$, if $G$ is of type $B_l$, $C_l$ or $D_l$.
\end{abstract}
\thispagestyle{empty}


\pagenumbering{arabic}

\section{Introduction}

Let $K$ be an  algebraically closed field of characteristic $p > 0$. For every  simply connected simple linear algebraic group $G$ over $K$ of rank $l$, the irreducible $KG$-modules with dimension below a bound proportional to $l^2$ were determined by Liebeck in \cite{liebeck}. Lübeck \cite{lubeck} extended these results taking a bound proportional to $l^3$. For groups of type $A_l$, this bound was $l^3/8$; for types $B_l$, $C_l$ and $D_l$, the bound was $l^3$. Extending this classification further is desirable for some applications (see for example \cite{halasi,lee}). The bound we take here is $\binom{l+1}{4}$ if $G$ is of type $A_l$, and $16 \binom{l}{4}$ if $G$ is of type $B_l$, $C_l$ or $D_l$.  

The irreducible $KG$-modules are parameterised by dominant weights $\l$, we denote them by $L(\l)$. Due to Steinberg's tensor product theorem, we need only consider the case where $\l$ is $p$-restricted. For small ranks ($l\le 20$ if $G$ has type $A_l$ and $l \le 11$ if $G$ has type $B_l$, $C_l$ or $D_l$), lists of weights $\l$ with $\dim L(\l)$ under the bound we consider can be found in \cite{lubpage}. There, similar lists for groups of exceptional type are also provided. We only consider groups of classical type.

Our results are summarised in the following two theorems. Throughout, $\epsilon_p(k)$ will denote 1 if $p$ divides $k$ and $0$ otherwise.\\

\begin{theorem} \label{thmA}
Let $G$ be a simply connected simple algebraic group of type $A_l$ and let $l \ge 9$. Table \ref{tableA} contains all nonzero $p$-restricted dominant weights $\l$ up to duals such that $\dim L(\l) < \binom{l+1}{4}$, as well as the dimensions of the corresponding modules $L(\l)$.\\
\end{theorem}

\begin{theorem} \label{thmBCD}
Let $G$ be a simply connected simple algebraic group of type $B_l$, $C_l$ or $D_l$ and let $l \ge 9$. Tables \ref{tableB}, \ref{tableC} and \ref{tableD} contain all nonzero $p$-restricted dominant weights $\l$ such that $\dim L(\l) < 16 \binom{l}{4}$, as well as the dimensions of the corresponding modules $L(\l)$. Note that for $p=2$ the modules for type $B_l$ and for type $C_l$ have the same dimensions; we only list them in Table \ref{tableC}.
\end{theorem}

\begin{table}[H]
  \setlength{\tabcolsep}{2em}
  \renewcommand{\arraystretch}{1.4}
  \begin{tabular}{ l l } 
    \specialrule{.1em}{.05em}{.05em} 
    $\l$ & $\dim L(\l)$ \phantom{I just would like the tables to be equal in length.} \\ 
    \specialrule{.1em}{.05em}{.05em} 
    $\l_1$ & $l+1$ \\ 
    $\l_2$& $\binom{l+1}{2}$ \\ 
    $2\l_1$ & $\binom{l+2}{2}$\\ 
    $\l_1+\l_l$ & $(l+1)^2-1-\epsilon_p(l+1)$\\ 
    $\l_3$ & $\binom{l+1}{3}$\\ 
    $3\l_1$ & $\binom{l+3}{3}$ \\ 
    $\l_1+\l_2$ & $2\binom{l+2}{3}- \epsilon_p(3)\binom{l+1}{3}$\\ 
    $\l_1+\l_{l-1}$& $3\binom{l+2}{3}-\binom{l+2}{2}-\epsilon_p(l)(l+1)$\\ 
    $2\l_1+\l_{l}$& $3\binom{l+2}{3}+\binom{l+1}{2}-\epsilon_p(l+2)(l+1)$ \\ 
    \specialrule{.1em}{.05em}{.05em} 
  \end{tabular}
  \caption{Type $A_l$}
  Nonzero $p$-restricted dominant weights $\l$ such that $\dim L(\l) < \binom{l+1}{4}$ for $l \ge 9$.
  \label{tableA}
\vspace{1em}
  
  \setlength{\tabcolsep}{2em}
  \renewcommand{\arraystretch}{1.4}
  \begin{tabular}{ l l } 
    \specialrule{.1em}{.05em}{.05em} 
    $\l$ & $\dim L(\l)$\phantom{I just would like the tables to be equal in length.}\\ 
    \specialrule{.1em}{.05em}{.05em} 
    $\l_1$ & $2l+1$ \\ 
    $\l_2$& $\binom{2l+1}{2}$ \\ 
    $2\l_1$ & $\binom{2l+2}{2}-\epsilon_p(2l+1)$ \\ 
    $\l_3$ & $\binom{2l+1}{3}$ \\ 
    $3\l_1$& $\binom{2l+3}{3}-(2l+1)-\epsilon_p(2l+3)(2l+1)$\\ 
    $\l_1+\l_2$ & $2^4 \binom{l+\frac{3}{2}}{3}-\epsilon_p(l)(2l+1)-\epsilon_p(3)(\binom{2l+1}{3})$ \\ 
    $\l_{l}$ ($l\le 13$) & $2^{l}$ \\ 
    \specialrule{.1em}{.05em}{.05em} 
  \end{tabular}
  \caption{Type $B_l$, $p\neq 2$}
  Nonzero $p$-restricted dominant weights $\l$ such that $\dim L(\l) < 16 \binom{l}{4}$ for $l \ge 9$.
  \label{tableB}

\vspace{1em}

  \setlength{\tabcolsep}{2em}
  \renewcommand{\arraystretch}{1.4}
  \begin{tabular}{ l l } 
    \specialrule{.1em}{.05em}{.05em} 
    $\l$ & $\dim L(\l)$\\  
    \specialrule{.1em}{.05em}{.05em} 
    $\l_1$ & $2l$ \\ 
    $\l_2$& $\binom{2l}{2}-2l-\epsilon_{p}(l)$\\ 
    $2\l_1$ & $\binom{2l+1}{2}$ \\ 
    $\l_3$ & $\binom{2l}{3}-2l-\epsilon_p(l-1)(2l)$\\
    $3\l_1$& $\binom{2l+2}{3}$\\ 
    $\l_1+\l_2$ & $2^4 \binom{l+1}{3}-\epsilon_p(2l+1)(1-\epsilon_p(3))(2l)-\epsilon_p(3)(\binom{2l}{3}-2l)$ \\ 
    $\l_{l}$ ($l\le 13$, $p=2$)& $2^{l}$  \\ 
    \specialrule{.1em}{.05em}{.05em} 
    
  \end{tabular}
  \caption{Type $C_l$}
  Nonzero $p$-restricted dominant weights $\l$ such that $\dim L(\l) < 16 \binom{l}{4}$ for $l \ge 9$.
  \label{tableC}
\end{table}

\begin{table}
  \setlength{\tabcolsep}{2em}
  \renewcommand{\arraystretch}{1.4}
  \begin{tabular}{ l l } 
    \specialrule{.1em}{.05em}{.05em} 
    $\l$ & $\dim L(\l)$\phantom{I just would like the tables to be equal in lengt} \\ 
    \specialrule{.1em}{.05em}{.05em} 
    $\l_1$ & $2l$ \\ 
    $\l_2$& $\binom{2l}{2}-\epsilon_{p}(2)(1+\epsilon_{p}(l))$ \\ 
    $2\l_1$ & $\binom{2l+1}{2}-1-\epsilon_p(l)$ \\ 
    $\l_3$ & $\binom{2l}{3}-\epsilon_p(2)(1+\epsilon_p(l))(2l)$ \\
    $3\l_1$& $\binom{2l+2}{3}-2l-\epsilon_{p}(l+1)(2l)$\\ 
    $\l_1+\l_2$ &  $2^4 \binom{l+1}{3}-\epsilon_p(2l-1)(2l)-\epsilon_p(3)\binom{2l}{3}$\\ 
    $\l_{l-1}$ ($l \le 15$) & $2^{l-1}$\\ 
    $\l_{l}$ ($l \le 15$) & $2^{l-1}$ \\
    \specialrule{.1em}{.05em}{.05em} 
  \end{tabular}
  \caption{Type $D_l$}
  Nonzero $p$-restricted dominant weights $\l$ such that $\dim L(\l) < 16 \binom{l}{4}$ for $l \ge 9$.
  \label{tableD}
\end{table}

\begin{remark}
The bounds have been taken so that (using the notation in Section \ref{reduction}) the tables include exactly the $\l$ with $\kappa(\l)\le 3$. In particular, they should exclude the weight $\l_4$, and in fact, $| \mathscr{W} \l_4 | = \binom{l+1}{4},16\binom{l}{4}$ respectively if $G$ has type $A_l$ or one of the other types.  
\end{remark}

In Section \ref{reduction} it is shown that Tables \ref{tableA}, \ref{tableB}, \ref{tableC} and \ref{tableD} contain all the weights that need to be considered. The dimensions of the modules in the tables are established in Section \ref{dimensions}.

\section*{Acknowledgements}
Above all I would like to thank my supervisor Professor Martin Liebeck for his encouragement,  guidance and precious comments. I am also thankful for the financial support of the Undergraduate Research Opportunities Program at Imperial College London.

\section{Preliminaries} \label{prelim}
Let $p$ and $K$ be as in the introduction and let $G$ be a simply connected cover of a simple classical algebraic group of rank $l$ over $K$. Let $B=UT$ be a Borel subgroup containing the maximal torus $T$ and with unipotent radical $U$, and let $B^-$ be the opposite Borel subgroup. Denote by $X(T),Y(T)$ respectively the character and cocharacter groups of $T$. Let $\Phi \subset Y(T)$ be the root system of $G$ and denote by $S=\{\a_1,...,\a_l\}\subset \Phi$ the base of simple roots for $B$, where we label Dynkin diagrams as in \cite{bourbaki}. We denote by $\l_1,...,\l_l$ the corresponding fundamental weights with respect to the usual pairing on $X(T)\times Y(T)$. The Weyl group $\mathscr{W}=N_G(T)/T$ is generated by the simple reflections $s_\a$ associated to the simple roots $\a \in S$. We denote by $w_0 \in \mathscr{W}$ the longest element of the Weyl group.\\
We recall some standard facts. The irreducible $KG$-modules are parameterised by dominant weights $\l \in X(T)$, that is, weights of the form $\l=\sum_{i=1}a_i \l_i$ with all $a_i\ge 0$. We denote the irreducible module with highest weight $\l$ by $L(\l)$. Given a $KG$-module $M$, we say that the element $\mu \in X(T)$ is a weight of $M$ if and only if the weight space $M_\mu=\{m \in M: tm=\mu(t)m\text{ for all } t \in T\}$ is nonzero, and we say that the multiplicity of $\mu$ in $M$ is $\dim M_\mu$. If $\l$ is dominant, we denote the multiplicity of $\mu$ in $L(\l)$ by $m_\l(\mu)$. A dominant weight as above is $p$-restricted if each $a_i$ satisfies $0\le a_i <p$. Steinberg's tensor product theorem \cite{steinberg} allows one to express any $L(\l)$ as a tensor product of twists of $KG$-modules with $p$-restricted highest weights, therefore we only consider $p$-restricted dominant weights.\\
In order to understand the module $L(\l)$, it can be useful to understand the related induced module $H^0(\l)$ and Weyl module $V(\l)$. In \cite{jantzen}, the group $G$ is regarded as a group scheme. Given a dominant $\l\in X(T)$ and the corresponding $B^-$-module $K_\l$, one constructs a left exact functor $\mathrm{ind}_{B^-}^G(-)$ whose derived functors are denoted by $H^i(\l)$. The induced module is simply $H^0(\l)$. In this framework, the Weyl module is defined thus: $V(\l)=H^0(-w_0\l)^*$. The induced module $H^0(\l)$ has a unique irreducible submodule isomorphic to $L(\l)$, that is, $L(\l)=\soc H^0(\l)$. Dually, $V(\l)$ has $L(\l)$ as its unique irreducible quotient, so we can also realise the latter as $L(\l)=V(\l)/\rad V(\l)$.\\
Write $e(\l) \in \Z[X(T)]$ for the element of the group ring of $X(T)$ corresponding to $\l\in X(T)$. Denote the formal character of a (finite-dimensional) $KG$-module $M$ by $\ch M=\sum_{\mu\in X(T)}\dim M_\mu e(\mu)$. Observe that $\ch M$ encodes all information about weight multiplicities of $M$ and in particular, it yields $\dim M$. For any dominant $\l\in X(T)$, one can compute $\ch V(\l)=\ch H^0(\l)$ using Weyl's character formula. There is an Euler characteristic defined for each $\mu \in X(T)$ as $\chi(\mu)=\sum_{i\ge0}(-1)^i \ch H^i(\l)$. If $\l$ is dominant, then Kempf's vanishing theorem states that $H^i(\l)=0$ for all $i>0$. Lastly, recall that both the $\chi(\l)=\ch H^0(\l)$ and the $\ch L(\l)$ with $\l$ dominant form $\Z$-bases of $\Z[X(T)]^\mathscr{W}$.

Finally, we remark that there is an isogeny $\varphi$ between the groups of type $B_l$ and $C_l$ when $p=2$, which is an isomorphism of abstract groups but not of algebraic groups. If $L(\l)$ is the irreducible module with highest weight $\l$ for one of the groups, then composing the action of the group with $\varphi$ yields the corresponding irreducible module for the other group. In particular, the weight multiplicities and dimensions are the same for both types. For this reason, we exclude the case $p=2$ in Table \ref{tableB}.

\section{Dimensional bounds and  reduction}\label{reduction}
We recall some basic weight theory. There is a partial order $\le$ on $X(T)$ defined by: $\mu \le \l$ if and only if $\l-\mu$ is a nonnegative linear combination of simple roots. Denote the (usual) action of the Weyl group on $X(T)$ by $(w,\mu)\mapsto w\mu$. Then $m_\l(\mu)=m_\l(w\mu)$. Assume $\l$ and $\mu$ are both dominant. If $\mu \le \l$ and $\mu \neq \l$, we say that $\mu$ is subdominant to $\l$. The weights of $L(\l)$ form a subset of the weights of $V(\l)$, and every dominant weight of $V(\l)$ is subdominant to $\l$. It follows that $\dim L(\l)=\sum_{\mu\le \l}m_\l(\mu)|W\mu|$, where $\mu$ runs over all dominant weights. The stabiliser of $\mu$ in $\mathscr{W}$ is the subgroup $\mathscr{W}_\mu \le \mathscr{W}$ generated by the reflections $s_\a$ such that $\langle \mu, \a \rangle =0$. Hence if $\mu$ is a weight of $L(\l)$ we have the bound
\begin{equation}\label{basic}
    \dim L(\l) \ge |\mathscr{W}\mu|=|\mathscr{W}:\mathscr{W}_\mu|.
\end{equation}
Finally, Premet's theorem \cite{premet} asserts that if $(G,p)$ is not special and $\l$ is $p$-restricted, then any $\mu\le \l$ is a weight of $L(\l)$. For the classical types, the pair $(G,p)$ is special only if $p=2$ and $G$ has type $B_l$ or $C_l$.

We now proceed to show that any dominant weight $\l$ not in Tables \ref{tableA}, \ref{tableB}, \ref{tableC} and \ref{tableD} must satisfy $\dim L(\l)\ge \binom{l+1}{4} $ if $G$ has type $A_l$ or $\dim L(\l)\ge 16\binom{l}{4} $ if $G$ has type $B_l$, $C_l$ or $D_l$. Following \cite{magaard}, for a $p$-restricted dominant weight $\mu =\sum_{i=1}^l a_i \l_i \in X(T)$, we define the integers
\[   
\kappa(\mu) = 
     \begin{cases}
       \sum_{i=1}^l i a_i & \quad\text{if }G\text{ is of type }B_l, C_l\text{ or }D_l\\
       \sum_{i=1}^l \min \{i,l+1-i\}a_i & \quad\text{if }G\text{ is of type }A_l\\
     \end{cases}
\]
\[   
r_{\mu} = 
     \begin{cases}
       0 & \quad\text{if }a_c=0\text{ for all }c> \frac{l+1}{2}\\
       \max \{c : 1 \le c < \frac{l+1}{2} \text{ and }a_{l+1-c}\neq 0\} & \quad\text{otherwise.}\\
     \end{cases}
\]

\begin{proposition}\label{abound}
Let $G$ be of type $A_l$ and let $\l$ be a $p$-restricted dominant weight. Set $\kappa := \kappa(\l)$. Then,
\[
\dim L(\l) \ge\begin{cases}
 \binom{l+1}{\kappa} \quad\text{if $\kappa \le \frac{l+1}{2}$,}\\
 \binom{l+1}{\lfloor l/2\rfloor} \quad\text{otherwise.} 
\end{cases}
\]
\end{proposition}
\begin{proof}
Both bounds are respectively part of Proposition 4.7 and Lemma 4.8 in \cite{magaard}.
\end{proof}
\begin{corollary}
Let $G$ be of type $A_l$, and assume $l\ge 9$. Any nonzero $p$-restricted dominant weight $\l \in X(T)$ not listed in Table \ref{tableA} satisfies $\dim L(\l)\ge \binom{l+1}{4}$.
\end{corollary}
\begin{proof}
The result follows from Proposition \ref{abound} and the observation that Table \ref{tableA} contains precisely the nonzero $p$-restricted dominant weights $\l$ with $\kappa(\l)\le 3$. 
\end{proof}
The following can be seen as an analogue of Proposition \ref{abound} for types $B_l$, $C_l$ and $D_l$.\\

\begin{proposition}\label{almost}
Let $G$ be of type $B_l$, $C_l$ or $D_l$, assume $l\ge 7$ and let $\l$ be a $p$-restricted dominant weight $\sum_{i=1}^l a_i\l_i$. Set $\kappa := \kappa(\l)$. Then, the following hold.
\begin{enumerate}[label=(\alph*)]
    \item Assume $r_\l\neq 0$ and, if $G$ has type $D_l$, assume $r_\l\ge3$. Then $| \mathscr{W}\l| \ge 2^{l-r_\l+1}\binom{l}{r_\l-1}$. In particular, $\dim L(\l) \ge 2^{l-r_\l+1}\binom{l}{r_\l-1}$.
    \item If $r_\l = 0$ and $\kappa \le (l+1)/2$, then $\dim L(\l) \ge 2^\kappa \binom{l}{\kappa}$.
    \item If $r_\l=0$ and $\kappa>(l+1)/2$, then $\dim L(\l) \ge 2^{\left\lfloor{\frac{l+2}{2}} \right\rfloor}\binom{l}{\left\lfloor{\frac{l+2}{2}} \right\rfloor}$.
\end{enumerate}
\end{proposition}
\begin{proof}
For part (a), note that the stabiliser $\mathscr{W}_\l$ is contained in $\mathscr{W}_{\l_{l-r_\l+1}}$ and use Bound (\ref{basic}). Part (b) follows from Proposition 4.7 in \cite{magaard} and the observation that $(l+1)/2\le l-3$. For (c), we argue as in the proof of 4.9(c) in \cite{magaard}. Define $d=\max\{i : a_i\neq 0\}$. We consider two cases.\\
 If $a_d=1$, then Lemma 4.5 from \cite{magaard} ensures that $L(\l)$ has a subdominant weight (with nonzero multiplicity) of the form $\sum_{i=1}^{\left \lfloor \frac{l+1}{2} \right \rfloor}b_i \l_i+\l_{\left\lfloor \frac{l+1}{2} \right\rfloor+1}$ and so by (a), we have $\dim L(\l)\ge |\mathscr{W}\mu|\ge 2^{\left\lfloor{\frac{l+2}{2}} \right\rfloor}\binom{l}{\left\lfloor{\frac{l+2}{2}} \right\rfloor}$. Otherwise if $a_d > 1$, then $\mu'=\mu-\a_d=\sum_{i=1}^{\left \lfloor \frac{l+1}{2} \right \rfloor}c_i \l_i+\l_{d+1}$ is subdominant to $\l$. By the previous case, there exists in turn some $\nu\le \mu'$ of the form $\sum_{i=1}^{\left \lfloor \frac{l+1}{2} \right \rfloor}b_i' \l_i+\l_{\left\lfloor \frac{l+1}{2} \right\rfloor+1}$. To see that $\nu$ has nonzero multiplicity in $L(\l)$, observe that $\l$ is $p$-restricted and $a_d>1$, hence $p>2$ and Premet's theorem applies. Again, applying (a) to $\nu$ yields the desired inequality.
\end{proof}
\begin{corollary}
Let $G$ be of type $B_l$, $C_l$ or $D_l$ and assume $l\ge 9$. Any nonzero $p$-restricted dominant weight $\l \in X(T)$ not listed in Tables \ref{tableB}, \ref{tableC} and \ref{tableD} satisfies $\dim L(\l)\ge 16\binom{l}{4}$.
\end{corollary}
\begin{proof}
Write $\l= \sum_{i=1}^l a_i\l_i$. First note that all the weights with $\kappa(\l)< 4$ appear in the three tables, so assume $\kappa(\l)\ge 4$. Now, if $r_\l=0$, the result directly follows from (b) and (c) of Proposition \ref{almost}, so we further assume $r_\l\neq 0$. If $r_\l \ge 2$ or, if $G$ has type $D_l$, if $r_\l \ge 3$, then Proposition \ref{almost}(a) shows that $\dim L(\l) \ge 2^{l-r_\l+1}\binom{l}{r_\l-1}$. An elementary check shows that this is greater or equal to $ 2^4 \binom{l}{4}$ for all $l\ge 9$. We discuss the remaining possibilities separately.\\
\underline{Types $B_l$ and $C_l$, $r_\l=1$}\\
In view of the tables, the weight $\l=\l_l$ needs only be considered when $G$ has type $C_l$ and $p \neq 2$. In this case by Premet's theorem $\l_l-(\a_{l-1}+\a_l)=\l_{l-2}$ has nonzero multiplicity in $L(\l)$ so Proposition \ref{almost}(a) yields the result. Next, if $a_j>0$ for some $j<l$, the stabiliser $\mathscr{W}_\l \le \mathscr{W}$ is a subgroup of $\mathscr{W}_{\l_j+\l_l}$. But then in both types $|\mathscr{W}\l|\ge |\mathscr{W}(\l_j+\l_l)|= 2^l\binom{l}{j}\ge 2^ll\ge 2^4\binom{l}{4}$. Finally, if $a_l>1$, then $p\neq 2$ and by Premet's theorem $\mu=\l-\a_{l-1}$ is a subdominant weight for both types; but then $r_\mu=2$ and Proposition \ref{almost}(a) yields the inequality.\\
\underline{Type $D_l$, $r_\l=1,2$}\\
We do not consider $\l=\l_{l}, \l_{l-1}$ as they appear in Table \ref{tableD}. If $a_j\neq 0$ for some $j<l-1$, we see that in all cases $\mathscr{W}_\l \subset \mathscr{W}_{\l_j+\l_l}$, so similarly as for $B_l$ and $C_l$, we get $|\mathscr{W}\l| \ge 2^{l-1}\binom{l}{j} \ge 2^{l-1}l\ge 2^4 \binom{l}{4}$, since also $l \ge 9$. We thus assume that $\l$ is of the form $a_{l-1}\l_{l-1}+a_l\l_{l}$. Note that by Premet's theorem and bound (\ref{basic}) it is enough to exhibit a weight $\mu\le \l$ whose orbit has size greater than $2^4\binom{l}{4}$. If $a_{l-1}a_l \neq 0$, set $\mu=\l-(\a_{l-2}+\a_{l-1}+\a_l)=\l_{l-3}+(a_{l-1}-1)\l_{l-1}+(a_l-1)\l_l$; if $a_l=0$ and $a_{l-1}>1$, set $\mu=\l-\a_{l-1}=\l_{l-2}+(a_{l-1}-2)\l_{l-1}$; and if $a_{l-1}=0$ and $a_l>1$, set $\mu=\l-\a_l=\l_{l-2}+(a_l-2)\l_l$. Then $r_\mu= 3$ and we conclude with another application of Proposition \ref{almost}(a).
\end{proof}
\begin{remark}\label{remark}
The weights considered in Theorem 5.1 in \cite{lubeck} are precisely the $p$-restricted weights $\l$ with $\kappa(\l)\le 2$.
\end{remark}
\section{Dimensions of irreducible $KG$-modules}\label{dimensions}
An effective method to compute the multiplicities of the weight spaces in characteristic $p$ was already observed in \cite{burgoyne}. One considers a certain bilinear form on $V(\l)_\Z$, a minimal admissible $\Z$-lattice of the Weyl module. Restricting this form to the weight space of $\mu \le \l$, reducing it modulo $p$ and computing its rank yields the multiplicity $m_\l(\mu)$. However, these computations can be lengthy and may not a priori provide much structural information about $V(\l)$ or $L(\l)$. For some cases we will instead find the constituents of modules having $L(\l)$ as a composition factor. Given a $KG$-module $M$, we write $M=N_1\mid N_2 \mid \cdots \mid N_k$ to indicate that $M$ has a filtration $M=M_1\supset M_2 \supset \cdots \supset M_{k+1}=0$ such that $M_i/M_{i+1} \cong N_i$ for each $i=1,...,k$. Note that the $N_i$ are not required to be irreducible.

The following tool provides information about $\rad V(\l)$, and it can be interpreted as providing the determinants of the bilinear forms above (this is explained in detail in II.8.17 of \cite{jantzen}). Denote by $H^0_\Z(\l)$ and $V_\Z(\l)$ the induced and Weyl modules over $\Z$. Then in II.8.16 of \cite{jantzen} one defines a homomorphism $V_\Z(\l)\rightarrow H^0_\Z(\l)$ which we denote by $T_\l$ such that $\mathrm{Im}(T_\l\otimes_\Z 1_K)=L(\l)$. Now, let $D$ be the group of divisors of $\Z$, that is, the abelian group generated by the formal elements $[q]$ for each prime $q$. Given a finitely generated torsion abelian group $N$ and a prime $q$, denote by $\nu_q(N)$ the composition length of $N\otimes_\Z \Z_{(q)}$ as a $\Z_{(q)}$-module. For a $G_\Z$-module $M$, define 
\[\nu^c(M)=\sum_{\mu \in X(T)} \nu(M_\mu)e(\mu)\in D\otimes_\Z\Z[X]^\mathscr{W}\]
where $\nu(M_\mu)=\sum_{q \text{ prime}}\nu_q(M_\mu)[q]\in D$.\\

Then, writing $\nu^c(T_\l)$ for $\nu^c(\mathrm{coker}(T_\l))$, for each subdominant $\mu\le \l$ one has that $L(\mu)$ is a composition factor of $\rad V(\l)$ if and only if the coefficient of $[p]\ch L(\mu)$ in $\nu^c(T_\l)$ is nonzero (recall that the $[p]\ch L(\mu)$ form a basis of $X(T)^\mathscr{W}$). In fact, the coefficient of $[p]e(\mu)$ in $\nu^c(T_\l)$ is the $p$-adic valuation of the determinant of the bilinear form mentioned above restricted to the weight space of $\mu$. If this coefficient is 1 then the composition factor $L(\mu)$ has multiplicity one in $V(\l)$, but the converse does not hold in general. We remark that for certain dominant weights $\l$, the character $\nu^c(T_\l)$ has been evaluated for arbitrary rank (see e.g. \cite{jantzenthesis}, \cite{mcninch}).

We now establish the dimension of $L(\l)$ for each $\l$ in Tables \ref{tableA}, \ref{tableB}, \ref{tableC} and \ref{tableD}. Note that by Remark \ref{remark}, we only need to consider the weights with $\kappa(\l)>2$. 
\subsection{Type $A_l$}
Note first that if $G$ is of type $A_l$ the weights of the form $\l_k$, $k\l_1$ correspond respectively to the exterior power ($\dim L(\l_k)=\binom{l+1}{k}$) and symmetric power ($\dim L(k\l_1)=\binom{l+k}{k}$) of the natural module, which are irreducible (as $\l$ is $p$-restricted).
The remaining weights $\l$ in Table \ref{tableA} with $\kappa(\l)>2$ are $\l_1+\l_2$, $\l_1+\l_{l-1}$ and $2\l_1+\l_l$. We will make use of the following result, which is part of 8.6 of \cite{seitz}.\\
\begin{lemma}\label{lmseitz}
    Let $\l=a_i \l_i+a_j \l_j$, $i<j$, be $p$-restricted with $a_i a_j\neq 0$. Suppose $\mu=\l-(\alpha_i+...+\alpha_j)$. Then $m_{\lambda}(\mu)=j-i+1-\epsilon_p(a_i+a_j+j-i)$.
\end{lemma}

The dimensions stated in Table \ref{tableA} easily follow from this, as the only subdominant weights to those listed either satisfy the conditions of the lemma or they have multiplicity one in the Weyl module (hence in $L(\l)$ by Premet's theorem). We give as an example the weight $\l=2\l_1+\l_l$; the other cases can be dealt with in a similar fashion. The subdominant weights are $\l-\alpha_1=\l_2+\l_l$ and $\l-(\alpha_1+...+\alpha_l)=\l_1$. The multiplicity of $\l-\alpha_1$ in the Weyl module is 1, hence so it is in $L(\l)$. By Lemma \ref{lmseitz} the subdominant weight $\l-(\alpha_1+...+\alpha_l)$ has multiplicity $l-\epsilon_p(l+2)$. This implies $\ch L(\l)=\chi(\l)-\epsilon_p(l+2)\chi(\mu)$ and therefore $\dim L(\l)=3\binom{l+2}{3}+\binom{l+1}{2}-\epsilon_p(l+2)(l+1)$.
\subsection{Types $B_l$, $C_l$ and $D_l$}
We now consider $G$ of type $B_l$, $C_l$ and $D_l$. By Remark \ref{remark}, the weights that need to be considered are the ones with $\kappa(\l)>2$ in Tables \ref{tableB}, \ref{tableC} and \ref{tableD}, that is, the weights $3\l_1$, $\l_3$, $\l_1+\l_2$, $\l_{l-1}$ and $\l_{l}$. 
The stated dimensions for the module $L(3\l_1)$ immediately follow from Proposition 4.7.4 in \cite{mcninch} if $G$ is of type $B_l$ or $D_l$. For $G$ of type $C_l$, consider the natural embedding of $G$ into a group $\tilde{G}$ of type $A_{2l-1}$. The Weyl module with highest weight $3\l_1$ is a irreducible $K\tilde{G}$-module and it remains irreducible for $G$ by 8.1(c) of \cite{seitz} (note that $\l$ is $p$-restricted). Similarly, 8.1(a) and 8.1(b) of \cite{seitz} show that if $p \neq 2$, the module $V(\l_3)$ is irreducible for $G$ of type $B_l$ or $D_l$. For $p=2$ and $G$ of type $D_l$, the dimension of $L(\l_3)$ is found in 7.2.5 of \cite{cavallin}. For $G$ of type $C_l$ and any $p$ (so in particular, for type $B_l$ and $p=2$), the dimension of $L(\l_3)$ follows from 4.8.2 in \cite{mcninch}.\\
Next, the weight $\l=\l_{l}$ is minuscule (i.e. it has no subdominant weights) for $G$ of types $B_l$ and $D_l$; in addition, if $G$ has type $D_l$, then $\l_{l-1}$ is minuscule too. Hence $V(\l)$ is irreducible in these cases. Clearly, if $\l$ is minuscule then $\dim L(\l)$ is just $|\mathscr{W}\l|$. \\
The only remaining weight is $\l_1 + \l_2$. The following is a direct consequence of 4.9.2 in \cite{mcninch}.\\
\begin{proposition}\label{mcninchalmost}
Let $\l=\l_1+\l_2$ and assume $l\ge5$ and $p>3$. Set $t=l,2l+1,2l-1$ respectively if $G$ has type $B_l$, $C_l$ or $D_l$. Then $\ch L(\l)=\chi(\l)-\epsilon_p(t)\chi(\l_1)$.
\end{proposition}
McNinch also provides the following computation for $l\ge 5$ (4.5.7 in \cite{mcninch}). Here, for any positive integer $k=q_1^{b_1}\dots q_r^{b_r}$ (where $q_i$ are primes and each $b_i$ is a positive integer), the divisor of $k$ is defined as $\mathrm{div}(k)=\sum_{i=1}^rb_i[q_i]$. Let $\l=\l_1+\l_2$.
\begin{equation} \label{mccomputation}
\nu^c(T_\l)=
    \begin{cases}
   \Div(3)\chi(\l_3)+\Div (2l)\chi(\l_1)+\Div(2)(\chi (2\l_1)+\chi(\l_2)-\chi(0)) \quad\text{if }G\text{ is of type }B_l,\\
    \Div(3)\chi(\l_3)+\Div (2l+1)\chi(\l_1) \quad\text{if }G\text{ is of type }C_l,\\
    \Div(3)\chi(\l_3)+\Div (2l-1)\chi(\l_1) \quad\text{if }G\text{ is of type }D_l.\\
    \end{cases}
\end{equation}

Since for types $C_l$ and $D_l$, the coefficient of $[2]$ in $\nu^c(T_\l)$ is zero, it follows that for these types the Weyl module $V(\l_1+\l_2)$ is irreducible if $p=2$. Clearly $\dim L(\l_1+\l_2)$ is also determined for type $B_l$ and $p=2$. Now set $t=2l,2l+1,2l-1$ respectively if $G$ has type $B_l$, $C_l$ or $D_l$. In view of the computation of $\nu^c(T_\l)$, we see that if $3\nmid t$, then $\rad V(\l_1+\l_2)=L(\l_3)$ and the dimensions in the tables for this case follow too.

Observe that Equation (\ref{mccomputation}) is however not enough to determine $\dim L(\l_1+\l_2)$ when $p=3$ and $3\mid t$. Instead, we will realise the module $L(\l_1+\l_2)$ as a composition factor of $S^3V$. In what follows, let $V$ be the natural module for $G$ and, for an integer $k\ge 1$, denote by $S^kV$ the $k$-th symmetric power of $V$. We write the elements of $S^kV$ as linear combinations of monomials $v_1v_2\cdots v_k$, where each $v_i \in V$.\\

\begin{lemma} \label{lemmasuprunenko}
Let $p>2$.
 \begin{enumerate}[label=(\alph*)]
     \item If $G=\mathrm{SL}_{l+1}(K)$ or $G=\mathrm{Sp}_{2l}(K)$, then $H^0(p\l_1)=S^pV=L((p-2)\l_1+\l_2)\mid L(p\l_1)$.
     \item If $G=\mathrm{SL}_{l+1}(K)$, then $H^0(p\l_l)=S^p(V^*)=L(\l_{l-1}+(p-2)\l_l)\mid L(p\l_l)$.\\
 \end{enumerate}
\end{lemma}
\begin{remark}
Note that $L(p\l_i) \cong L(\l_i)^{(p)}$.
\end{remark}
\begin{proof}
For (a), notice that when $\Char K=p$, the unique irreducible submodule of $H^0(p\l_1)=S^pV$ is $N=\{v^p:v \in V\}$, of highest weight $p\l_1$. This quotient is the $p$-th reduced symmetric power of $V$, which is irreducible for $G=\mathrm{SL}_l(K)$, $\mathrm{Sp}_{2l}(K)$ respectively by 1.2 and 2.2 of \cite{suprunenko}. Now observe that the highest weight of the quotient $S^pV/N$ is $(p-2)\l_1+\l_2$. Part (b) is analogous.
\end{proof}
\begin{corollary}\label{corollarysuprunenko}
If $G$ has type $C_l$ then $S^3V=L(\l_{1}+\l_2)\mid L(3\l_1)$.
\end{corollary}
The dimension of $L(\l_1+\l_2)$ in Table \ref{tableC} is now justified for all $p$. Note also that Lemma \ref{lemmasuprunenko} yields again the dimension of $L(\l_1+\l_2)$ in Table \ref{tableA} for $p=3$.

For types $B_l$ and $D_l$, we will use a result about type $A_l$. Let $G=\mathrm{SL}_{l+1}(K)$. We denote by $e_1,...,e_{l+1}$ the elements of the standard basis of $V$. We also write $e_1^*,...,e_{l+1}^*$ for the corresponding basis elements of the dual module $V^*$.\\

\begin{lemma}\label{decomposition}
Let $G=\mathrm{SL}_{l+1}(K)$ and assume $p=3$ and $3\nmid l+2$. 
\begin{enumerate}[label=(\alph*)]
    \item The module $S^2V\otimes V^*=L(\l_1) \mid L(2\l_1+\l_l) \mid L(\l_1)$ is indecomposable and $\soc (S^2V\otimes V^* )=\{\sum_{i=1}^{l+1} ve_i\otimes e_i^*: v\in V\}$.
    \item The module $S^2(V^*)\otimes V=L(\l_l) \mid L(\l_1+2\l_l) \mid L(\l_l)$ is indecomposable and $\soc (S^2(V^*)\otimes V )=\{\sum_{i=1}^{l+1} v^*e_i^*\otimes e_i: v^*\in V^*\}$.
\end{enumerate}

\end{lemma}
\begin{proof}
By Proposition 4.6.10 in \cite{mcninch} the modules are indecomposable and have the stated constituents. The fact that $soc (S^2V\otimes V^* )=\{\sum_{i=1}^l ve_i\otimes e_i^*: v\in V\}$ follows from the observation that it is a submodule isomorphic to $V$ and the fact that $S^2V\otimes V^*$ is indecomposable. The argument for $\soc (S^2(V^*)\otimes V )$ is analogous.
\end{proof}

In the following, instead of taking the simply connected group of type $B_l$ or $D_l$, we take $G$ to be the special orthogonal group $\mathrm{SO}_n(K)$ ($n=2l+1$, $2l$ respectively), realised as follows. Let $e_1,...,e_l,e_0,e_{-l},...,e_{-1}$ be the elements of the ordered standard basis of $\mathrm{SL}_{n}(K)$ (dropping $e_0$ for type $D_l$). We define $\mathrm{SO}_{n}(K)$ as the subgroup of $\mathrm{SL}_{n}(K)$ preserving the quadratic form $q(\sum x_i e_i)=\sum_{i=1}^nx_ix_{-i}+\frac{1}{2}x_0^2$ (dropping the term $\frac{1}{2}x_0^2$ for type $D_l$).

Given a $KG$-module, we denote by $M \downarrow H$ the restriction of $M$ to a subgroup $H$ of $G$.\\

\begin{proposition}\label{final}
Let $G=\mathrm{SO}_n (K)$, where $n=2l$ or $2l+1$, and set $p=3$.
\begin{enumerate}[label=(\alph*)]
    \item  Suppose $3 \nmid n-1$. Then $S^3V= L(\l_1+\l_2) \mid (L(3\l_1)\oplus L(\l_1))$.
    \item  Suppose $3 \mid n-1$. Then $S^3V=L(\l_1)\mid L(\l_{1}+\l_2) \mid (L(3\l_1)\oplus L(\l_1))$.
\end{enumerate}

\end{proposition}

\begin{proof}

Fix $B$ to be the Borel subgroup of upper triangular matrices in $G$. Define $Q\in S^2V$ as $Q=\frac{1}{2}e_0^2+\sum_{i=1}^l e_ie_{-i}$, dropping the term $\frac{1}{2}e_0^2$ if $n=2l$. Define also $J=\{vQ : v \in V\} \subset S^3V$. By II.2.18 of \cite{jantzen}, we have that $H^0(3\l_1)\cong S^3V/J$. Clearly $J\cong L(\l_1)$. Also, as in the proof of Lemma \ref{lemmasuprunenko}, since $\Char K=3$, we have the irreducible submodule $N=\{v^3 : v \in V\} \subset S^3V$, and again $N\cong L(3\l_1)$. Since $H^0(3\l_1)$ has a unique irreducible submodule and $N \cap J=0$, it follows that $\soc S^3V = N \oplus J$.

Now, to see (a), note that $3 \nmid n-1$ implies $3 \nmid t$, thus $\dim L(\l_1+\l_2)$ is known by the discussion after Equation (\ref{mccomputation}). Since $S^3V/{\soc S^3V}=S^3V/(N \oplus J)$ has a maximal vector with weight $\l_1+\l_2$ (namely, $e_1^2e_2+(N\oplus J)$) and it has the same dimension as $L(\l_1+\l_2)$, the result follows. For (b), we separate into cases $D_l$ and $B_l$.

\underline{Case $n=2l$}\\
Assume $3\mid 2l-1$. Let $H$ be the subgroup of type $A_{l-1}$ inside $G$ stabilising the subspace in $V$ spanned $e_1,...,e_l \in V$ as well as the subspace spanned by $e_{-l},...,e_{-1}$. Denote by $\tilde{V}$ the natural module for this subgroup, as well as $\tilde{L}(\mu)$ for the irreducible module of $H$ with highest weight $\mu$. The restriction $V \downarrow H =\tilde{V}\oplus \tilde{V}^*$ yields a decomposition
\begin{equation*}
    S^3V \downarrow H =S^3\V\oplus S^3(\V^*)\oplus(S^2\V\otimes \V^*)\oplus(S^2(\V^*)\otimes \V).
\end{equation*}
Visibly, $N$ corresponds to $\tilde{L}(3\l_1)\oplus \tilde{L}(3\l_l)\subset S^3\tilde{V} \oplus S^3(\tilde{V}^*)$. By Lemma \ref{lemmasuprunenko}, we have $(S^3\tilde{V} \oplus S^3(\tilde{V}^*))/N \cong \tilde{L}(\l_1+\l_2)\oplus \tilde{L}(\l_{l-1}+\l_{l})$. It follows that
\begin{equation*}
    (S^3V/N) \downarrow H =\tilde{L}(\l_1+\l_2)\oplus \tilde{L}(\l_{l-1}+\l_{l})\oplus(S^2\V\otimes \V^*)\oplus(S^2(\V^*)\otimes \V).
\end{equation*}

Next, note that as a $KH$-module, $J=J_1\oplus J_l$, where $J_1=\{\sum_{i=1}^lve_i\otimes e_{-i}: v \in \V\} \subset S^2V \otimes V^*$ and $J_l=\{\sum_{i=1}^lv^*e_{-i}\otimes e_i: v^* \in \V^*\} \subset S^2(\V^*) \otimes \V$. Lemma \ref{decomposition} shows that $J=\soc (S^2\V \otimes \V^*)\oplus \soc (S^2(\V^*) \otimes \V)\cong \tilde{L}(\l_1)\oplus \tilde{L}(\l_l)$. Write $M_1=(S^2\V \otimes \V^*)/J_1$ and $M_l=(S^2(\V^*) \otimes \V)/J_l$. Denoting $M=S^3V/(N\oplus J)$, we have
\begin{equation}\label{restriction3}
    M \downarrow H =\tilde{L}(\l_1+\l_2)\oplus \tilde{L}(\l_{l-1}+\l_{l})\oplus M_1\oplus M_l
\end{equation} where again by Lemma 4.6, $M_i$ is indecomposable and has composition factors $\tilde{L}(\l_i)\mid \tilde{L}(2\l_i+\l_{l-i+1})$ for $i=1,l$.\\
Now, let $I$ be a nonzero irreducible $KG$-submodule of $M$. Observe that the action of the antidiagonal element $s\in G$ gives (nonzero) linear maps $\tilde{L}(\l_1+\l_2)\leftrightarrow \tilde{L}(\l_{l-1}+\l_{l})$ and $M_1 \leftrightarrow M_l$. This shows that $I$ must contain at least one of $R_1=\tilde{L}(\l_1+\l_2)\oplus  \tilde{L}(\l_{l-1}+\l_{l})$ or $R_2=\tilde{L}(2\l_1+\l_{l})\oplus \tilde{L}(2\l_l+\l_{1})$ (as $M_1$, $M_l$ are indecomposable for $H$). We show that in fact $I$ must contain $R_1\oplus R_2$. Let $r\in G$ be the element sending $e_2\mapsto e_2+e_l$, $e_{-l} \mapsto e_{-l}-e_{-2}$ and fixing $e_j$ for every $j\neq 2,-l$. Define also $x_1=e_1^2e_2+(N\oplus J)\in R_1$ and $x_2=e_1^2e_l+(N\oplus J)\in R_2$. Then $rx_1=x_1+x_2=r^tx_2$, where $r^t\in G$ denotes the transpose of $r$. It follows that $R_1\oplus R_2 \subset I$. Now, note that $x_1\in I$ is a maximal vector with weight $\l_1+\l_2$, so that $I=L(\l_1+\l_2)$. In view of (\ref{restriction3}), it is clear that either $M=L(\l_1) \mid L(\l_1+\l_2)$ or $M=L(\l_1+\l_2)$. Now since $3\mid t$, Equation (\ref{mccomputation}) implies $\dim L(\l_1+\l_2) \le \dim H^0(\l_1+\l_2)-\dim L(\l_3)-\dim L(\l_1)=\dim M-2l$, so in fact $M=L(\l_1) \mid L(\l_1+\l_2)$.

\underline{Case $n=2l+1$}\\
Assume $3 \mid l$. We now consider $H$ to be the subgroup of type $D_{l}$ inside $G$ that fixes $e_0$, and as before denote by $\V$ its natural module, as well as $\L(\mu)$ for the irreducible $H$-module with highest weight $\mu$. The restriction $V \downarrow H=\tilde{V}\oplus (Ke_0)$, yields $S^3V \downarrow H =S^3(\V)\oplus S^2(\V)\oplus \V \oplus \tilde{T}$, where $\tilde{T}$ is trivial for $H$.
Now $S^3(\V)$ has composition factors as described in (a). Visibly, the submodule $N\subset S^3V$ corresponds to the direct sum of $\tilde{T}$ and the copy of $\L(3\l_1)$ inside $S^3(\V)$. Now by Proposition 4.7.4 in \cite{mcninch} and since $p\mid l$, we have that $S^2(\V)=\tilde{T}_1\mid \L(2\l_1)\mid \tilde{T}_2$ is indecomposable, where $\tilde{T}_1,\tilde{T}_2$ are trivial for $H$. Note also that $J \downarrow H$ decomposes as the direct sum of $\tilde{T}_2$ and the copy of $\L(\l_1)$ inside $S^3(\V)$. Combining these observations, we have $M=S^3V/(N \oplus J) \downarrow H = \L(\l_1+\l_2)\oplus M_0\oplus \V$, where $M_0=\tilde{T}_1 \mid \L(2\l_1)$ is indecomposable. 
As before, let $I\subset M$ be an irreducible $KG$-submodule. Note that $M$ has no trivial submodules. Comparing dimensions with Table \ref{tableB}, we see that $I$ must contain $\tilde{L}(\l_1+\l_2)$. Comparing now the dimension of $M/I$, we see that the only possibilities are $I=\L(\l_1+\l_2)\oplus \L(2\l_1)$ and $I=M$. Again $I$ contains a maximal vector with weight $\l_1+\l_2$, so that $I=L(\l_1+\l_2)$. By the same dimensional argument as for type $D_l$, $M=L(\l_1)\mid L(\l_1+\l_2)$.

\end{proof}
Proposition \ref{final} yields the remaining dimensions for the weight $\l_1+\l_2$ stated in Tables \ref{tableB} and \ref{tableD}. This concludes the proof of Theorem \ref{thmBCD}.
 
\setlength{\bibsep}{2pt}

\'Alvaro L. Mart\'inez, \emph{Imperial College London}, London SW7 2AZ, UK. \url{alm116@ic.ac.uk}

\end{document}